\documentclass[reqno]{llncs}

\usepackage{times,amsmath,amssymb,enumerate,gastex,rotating}

\newcommand{\mcA}{\mathcal{A}}
\newcommand{\mcB}{\mathcal{B}}
\def\cl{\mathsf{cl}}
\newcommand{\Cmc}{\mathcal{C}}

\newcommand{\dom}{\mathop{\mathrm{dom}}}
\newcommand{\DSPACE}{\mathsf{DSPACE}}
\newcommand{\NSPACE}{\mathsf{NSPACE}}
\newcommand{\eval}{\mathsf{val}}
\newcommand{\EXPTIME}{\mathsf{EXPTIME}}
\newcommand{\FIM}{\mathsf{FIM}}
\newcommand{\first}{\mathsf{first}}
\newcommand{\FG}{\mathsf{FG}}
\newcommand{\IRR}{\mathsf{IRR}}
\newcommand{\lab}{\mathsf{label}}
\newcommand{\last}{\mathsf{last}}
\newcommand{\LOGCFL}{\mathsf{LOGCFL}}
\newcommand{\lop}{\mathsf{loop}}
\newcommand{\MT}{\mathsf{MT}}
\newcommand{\nodes}{\mathsf{nodes}}
\newcommand{\out}{\mathsf{out}}
\newcommand{\rest}{\mathord\restriction}

\newcommand{\Ptime}{\mathsf{P}}
\newcommand{\POLYLOGSPACE}{\mathsf{POLYLOGSPACE}}
\newcommand{\PSPACE}{\mathsf{PSPACE}}

\newcommand{\supp}{\mathsf{sup}}
\newcommand{\dA}{\mathbb{A}}
\newcommand{\dB}{\mathbb{B}}
\newcommand{\dC}{\mathbb{C}}

\begin{document}

\title{Compressed word problems for inverse monoids}
\author{Markus Lohrey} \institute{Universit\"at
Leipzig, Institut f\"ur Informatik, Germany\\
\email{lohrey@informatik.uni-leipzig.de}}

\maketitle

\begin{abstract}
The compressed word problem for a finitely generated monoid $M$ asks whether two given compressed
words over the generators of $M$ represent the same element of $M$. For string compression,
straight-line programs, i.e., context-free grammars that generate a single string, are used in this 
paper. It is shown that the compressed word problem for a free inverse monoid of finite rank at least
two is complete for $\Pi^p_2$ (second universal level of the polynomial time hierarchy). Moreover, 
it is shown that there exists a fixed finite idempotent presentation (i.e.,  a finite set of relations involving
idempotents of a free inverse monoid), for which the corresponding quotient monoid has a $\PSPACE$-complete
compressed word problem. The ordinary uncompressed word problem for such a quotient can be 
solved in logspace \cite{LoOn06IC}.  Finally, a $\PSPACE$-algorithm that checks whether a given element
of a free inverse monoid belongs to a given rational subset is presented.  This problem is also shown
to be $\PSPACE$-complete (even for a fixed finitely generated
submonoid instead of a variable rational subset). 
\end{abstract}

\section{Introduction}

The decidability and complexity of algorithmic problems
in (finitely generated) monoids and groups is a classical topic at the 
borderline of computer science and mathematics.
The most basic question of this kind is the {\em word problem},
which asks whether two words over the generators represent the 
same element.  Markov \cite{Mar47} and Post  \cite{Po47} proved independently that the word 
problem for finitely presented monoids is undecidable in 
general.  Later, Novikov  \cite{Nov58} and Boone \cite{Boo59} extended the result of Markov and 
Post to finitely presented groups, see the the survey \cite{MaMeSa95} for further information.

In this paper, we are interested in {\em inverse monoids}. A monoid is inverse, if for each
element $x$ there exists a unique ``inverse'' $x^{-1}$ such that
$x = x x^{-1} x$ and $x^{-1} = x^{-1}xx^{-1}$ \cite{Law99}. 
In the same way as groups can be 
represented by sets of permutations, inverse monoids
can be represented by sets of partial injections \cite{Law99}. 
Algorithmic questions for inverse monoids
received increasing attention in the past and inverse monoid theory
found several applications in combinatorial group theory, see e.g. 
\cite{BiMaMe94,Chof91,ChAl98,DeMeSe05,DiLoOn08,margolis:1993,MeSa96,LoOn06IC,Sil92,Sil96} and the survey \cite{MaMeSa95}.

Since the class of inverse monoids forms a variety of algebras 
(with respect to the operations of multiplication, inversion, and the identity
element), the free inverse monoid $\FIM(\Gamma)$ 
generated by a set $\Gamma$ exists.
Munn gave in \cite{munn:1974} an explicit representation of the 
free inverse monoid $\FIM(\Gamma)$. Elements can be represented
by finite subtrees of the Cayley-graph of the free group generated 
by $\Gamma$ (so called {\em Munn trees}). Moreover, there are two distinguished nodes (an initial
node and a final node). Multiplication of two elements of 
$\FIM(\Gamma)$ amounts of gluing the two Munn trees together, where
the final node of the first Munn tree is identified with the initial node of the 
second Munn tree. This gives rise to a very simple algorithm for the word
problem of $\FIM(\Gamma)$, which can moreover implemented in linear time.
In \cite{LoOn06IC}, it was also shown (using Munn trees together with a result of Lipton and Zalcstein
\cite{LiZa77} saying that the word problem for a finitely generated free group can be solved in 
logspace) that the word problem for  $\FIM(\Gamma)$ can be solved in 
logspace.

Although the word problem for a free inverse monoid can be solved very efficiently,
there are several subtle differences between the algorithmic properties of 
free inverse monoids on the one hand and free monoids and free groups on the 
other hand. Let us give two examples:
\begin{itemize}
\item Solvability of equations: By the seminal results of Makanin, this problem is decidable
for free monoids \cite{Mak77} 
and free groups \cite{Mak82}. 
On the other hand, solvability of equations
in a finitely generated free inverse monoid of rank at least 2 (the rank is the minimal number of generators)
is undecidable \cite{Rozenlat85}.
\item Rational subset membership problem: Membership in a given rational subset of a free monoid
or free group can be decided in polynomial time. The same problem is {\sf NP}-complete
for finitely generated free inverse monoids of rank at least two \cite{DiLoMi08}.
\end{itemize}
In this paper, we show that in a certain sense also the word problem is harder for free inverse
monoids than free monoids (groups). For this we consider the {\em compressed word problem},
where the input words are given succinctly by so called {\em straight-line programs} (SLPs)  \cite{PlRy99}.
An SLP is a context free grammar that generates  only one word,
see Section~\ref{S SLP}.
Since the length of this word may grow exponentially with the size (number of 
productions) of the SLP, SLPs can be seen as a compact string
representation. SLPs turned out to be a very
flexible compressed representation of strings, which are well suited for
studying algorithms for compressed strings; see e.g.
\cite{BeChRa08,GaKaPlRy96,Lif07,Loh06siam,Loh11IC,MiShTa97,Pla94,PlaRy98}.

In the compressed word problem for a finitely generated monoid $M$ the input
consists of two SLPs that generate words over the generators of $M$, and it is asked
whether these two words represent the same element of $M$.
Hence, the compressed word problem for a free monoid simply asks, whether
two SLPs generate the same word. Plandowski proved in \cite{Pla94} that this problem
can be solved in polynomial time; the best algorithm is due to Lifshits \cite{Lif07} and 
has a cubic running time. Based on Plandowski's result, it was shown in 
\cite{Loh06siam} that the compressed word problem for a free group can be solved
in polynomial time. This result has algorithmic implications for the ordinary (uncompressed)
word problem:  In  \cite{LoSchl07,Schl06} it was shown that the word problem for the automorphism 
group of a group $G$ can be reduced in polynomial time to the 
{\em compressed} word problem for $G$ (more general: the word problem for 
the endomorphism monoid of a monoid $M$ can be reduced in polynomial time to the 
{\em compressed} word problem for $M$).  Hence, the word problem for
the  automorphism group of a free group turned out to be solvable in polynomial time \cite{Schl06}, which solved
an open problem from combinatorial group theory \cite{KMSS03}. Generalizations of this result for larger classes of groups can 
be found in \cite{LoSchl07,Macd09}.

Our first main result states that the compressed word problem for every finitely generated free inverse monoid 
of rank at least two is complete for $\Pi^p_2$, the second universal level of the polynomial time 
hierarchy (Thm.~\ref{thm:cwp-pi2p}). The upper bound follows easily using Munn's solution for the word problem together with 
the above mentioned result of Lipton and Zalcstein for free groups.  The lower bound is based on a reduction 
from a variant of the SUBSETSUM problem together with an encoding of a SUBSETSUM instance by
an SLP \cite{Loh06siam}. Hence, the compressed word problem for free inverse monoids is indeed computationally harder
than the compressed word problem for free monoids (groups) (unless $\mathsf{P} = \Pi^p_2$).
It is not difficult to see that the compressed word problem for a  free inverse monoid 
of rank 1 can be solved in polynomial time (Prop.~\ref{prop:rank-one}).

In \cite{margolis:1993}, Margolis and Meakin presented a large
class of finitely presented inverse monoids with
decidable word problems. An inverse monoid from
that class is of the form $\FIM(\Gamma)/P$, where $P$ is a presentation consisting
of a finite number of relations $e=f$, where $e$ and $f$
are idempotents of $\FIM(\Gamma)$; we call such a 
presentation idempotent. 
In fact, in \cite{margolis:1993} it is shown that even the uniform word
problem for idempotent presentations is decidable. In this 
problem, also the presentation is part of the input.
An alternative proof for the decidability result of Margolis and Meakin
was given in \cite{Sil92}.
In \cite{LoOn06IC} it was shown that the word problem for every inverse monoid 
$\FIM(\Gamma)/P$, where $P$ is an idempotent presentation, can be solved
in logspace. This implies that the compressed word problem for each of 
these inverse monoids belongs to the class $\PSPACE$. Our second main result states
that the are specific idempotent presentations $P$ such that the compressed word
problem for $\FIM(\Gamma)/P$ is $\PSPACE$-complete (Thm.~\ref{thm:PSPACE-idempot}).

In the last part of the paper we consider the compressed variant of the rational subset
membership problem. The class of rational subsets of a monoid $M$ is the smallest
class of subsets, which contains all finite subsets, and which is closed under union, product
and Kleene star ($A^*$ is the submonoid generated by the subset $A \subseteq M$). 
If $M$ is finitely generated by $\Gamma$, then a rational subset of $M$ can be represented
by a finite automaton over the alphabet $\Gamma$. In this case, the rational subset membership problem
asks, whether a given element of $M$ (given by a finite word over $\Gamma$) belongs to a given
rational subset (given by a finite automaton over $\Gamma$). Especially for groups, this
problem is intensively studied, see e.g. \cite{LohSt10,LohSt09tocs}. In \cite{DiLoMi08}, it was shown that the rational
subset membership problem for a free inverse monoid of finite rank at least two is {\sf NP}-complete.
Here, we consider the {\em compressed rational subset membership problem}, where the input
consists of an SLP-compressed word over the generators and a finite automaton over the generators.
We show that the compressed rational
subset membership problem for a free inverse monoid of finite rank at least two is {\sf PSPACE}-complete.
The difficult part of the proof is to show membership in $\PSPACE$.  
$\PSPACE$-hardness holds already for the case that the rational subset is a fixed
finitely generated submonoid (Thm.~\ref{PSPACE-hard submonoid}).

\section{Preliminaries}

Let $\Gamma$ be a finite alphabet.
The \emph{empty word} over $\Gamma$ is denoted by $\varepsilon$.
Let $s = a_1 \cdots a_n \in \Gamma^*$ be a word over $\Gamma$,
where $n\geq 0$ and $a_1, \ldots, a_n \in \Gamma$ for $1 \leq i \leq n$.
The \emph{length} of $s$ is $|s| = n$.
For $1 \leq i \leq n$ let $s[i] = a_i$ and
for $1 \leq i \leq j \leq n$ let $s[i,j] = a_i a_{i+1} \cdots a_j$.
If $i > j$ we set $s[i,j] = \varepsilon$.
For $n \in \mathbb{N}$ let 
$\Gamma^{\leq n} = \{w \in \Gamma^* \mid |w| \leq n\}$.
We write $s \preceq  t$ for $s,t \in \Gamma^*$, if $s$ is 
a prefix of $t$. A set $A \subseteq \Gamma^*$ is {\em prefix-closed},
if $u \preceq v \in A$ implies $u \in A$.
We denote with $\Gamma^{-1} = \{ a^{-1} \mid 
a \in \Gamma\}$ a disjoint copy of the finite alphabet $\Gamma$. 
For $a^{-1} \in \Gamma^{-1}$ we define
$(a^{-1})^{-1}=a$; thus, $^{-1}$ becomes an involution on the alphabet $\Gamma\cup\Gamma^{-1}$.
We extend this involution to words from $(\Gamma\cup\Gamma^{-1})^*$ by setting 
$(a_1  \cdots a_n)^{-1} = a_n^{-1} \cdots a_1^{-1}$, where 
$a_i \in \Gamma\cup\Gamma^{-1}$.
For $a \in \Gamma\cup\Gamma^{-1}$ 
and $n \geq 0$ we use $a^{-n}$ as an abbreviation
for the word $(a^{-1})^n$.

We use standard terminology from automata theory.
A {\em nondeterministic finite automaton} (NFA) over an input alphabet $\Gamma$ is a tuple
$\mcA = (Q, \Sigma, \delta, q_0, F)$, where $Q$ is the set of states,
$\delta \subseteq Q \times \Sigma \times Q$ is the 
transition relation, $q_0\in Q$ is the initial state, and $F \subseteq Q$
is the set of final states. For a {\em deterministic finite automaton}, 
$\delta : Q \times \Sigma \to_p Q$ is a partial mapping from
$Q \times \Sigma$ to $Q$. 

\paragraph{\bf Complexity theory:}
We assume some basic background
in complexity theory, see e.g. \cite{Papa94}. Recall that
$\Pi^p_2$ (the second universal level of the polynomial time 
hierarchy) is the class of all languages $L$ for which there
exists a polynomial time predicate $P(x,y,z)$ and a polynomial $p(n)$
such that
$$
L = \{ x \in \Sigma^* \mid \forall y \in \Sigma^{\leq p(|x|)} \exists z \in
\Sigma^{\leq p(|x|)} : P(x,y,z) \}.
$$
$\POLYLOGSPACE$ 
denotes the class $\NSPACE(\log(n)^{O(1)}) = \DSPACE(\log(n)^{O(1)})$.
A $\PSPACE$-transducer is a deterministic Turing machine with a 
read-only input tape, a write-only output tape 
and a  working tape, whose length is bounded by $n^{O(1)}$,
where $n$ is the input length. The output is written
from left to right on the output tape, i.e., 
in each step the transducer either outputs a new symbol
on the output tape, in which case the output head moves
one cell to the right, or the transducer does not output
a new symbol in which case the output head does not move.
Moreover, we assume that the transducer terminates for every
input. This implies that a $\PSPACE$-transducer computes a 
mapping $f : \Sigma^* \to \Theta^*$, where $|f(w)|$ is bounded
by $2^{|w|^{O(1)}}$. A $\POLYLOGSPACE$-transducer is defined
in the same way as a $\PSPACE$-transducer, except that 
the length of the working tape is bounded by $\log(n)^{O(1)}$.
The proof of the following lemma uses the same idea that shows that 
logspace reducibility is transitive.

\begin{lemma} \label{PSPACE}
Assume that $f : \Sigma^* \to \Theta^*$ 
can be computed by a $\PSPACE$-transducer and 
that  $g : \Theta^* \to \Delta^*$ can be computed
by a $\POLYLOGSPACE$-transducer. 
Then the mapping  $f \circ g$ can be computed by a 
$\PSPACE$-transducer.
In particular, if the language 
$L \subseteq \Theta^*$ belongs to $\POLYLOGSPACE$, then
$f^{-1}(L)$ belongs to $\PSPACE$.
\end{lemma}

\begin{proof}
The proof uses the same idea that shows that the composition
of two logspace computable mappings is again logspace computable.
Let $w \in \Sigma^*$ be an input.
Basically, we run the $\POLYLOGSPACE$-transducer for $g$
on the input $f(w)$. But since $f$ can be computed by a 
$\PSPACE$-transducer (which can generate an exponentially long 
output) the length of $f(w)$ can be only bounded by $2^{|w|^{O(1)}}$.
Hence, we cannot construct $f(w)$ explicitly. But this is not 
necessary. We only store  a pointer to some position in $f(w)$ 
(this pointer needs space $|w|^{O(1)}$) while
running the $\POLYLOGSPACE$-transducer for $g$. Each 
time, this algorithm needs the $i^{th}$ letter of $f(w)$, we 
run the $\PSPACE$-transducer for $L$ until the $i^{th}$ output symbol
is generated. Note that the $\POLYLOGSPACE$-transducer for $g$
needs space $\log(2^{|w|^{O(1)}})^{O(1)} = |w|^{O(1)}$ while running on 
$f(w)$. Hence, the total space requirement is bounded by $|w|^{O(1)}$.

The second statement of the lemma follows indeed from the first statement,
by taking the $\POLYLOGSPACE$-transducer $g = \chi_L : \Theta^* \to \{0,1\}$
(the characteristic function of $L$).
\qed
\end{proof}

\section{Free groups}

It is common to identify a congruence $\alpha$ on a monoid $M$
with the surjective homomorphism from $M$ to the quotient $M/\alpha$
that maps an element $m \in M$ to the congruence class of $m$ with respect to
$\alpha$.
The \emph{free group} $\FG(\Gamma)$ generated by the set $\Gamma$ is 
the quotient monoid 
\begin{equation}
\label{def-free-group}
\FG(\Gamma) = (\Gamma\cup\Gamma^{-1})^*/\delta,
\end{equation}
where $\delta$ is the smallest congruence on $(\Gamma\cup\Gamma^{-1})^*$ that contains 
all pairs $(bb^{-1}, \varepsilon)$ for $b \in \Gamma\cup\Gamma^{-1}$.
It is well known that for 
every $u \in (\Gamma\cup\Gamma^{-1})^*$ there exists a unique word
$r(u)\in (\Gamma\cup\Gamma^{-1})^*$ (the \emph{reduced normal form of $u$})
such that $\delta(u) = \delta(r(u))$ and $r(u)$ does not contain
a factor of the form $bb^{-1}$ for $b \in \Gamma\cup\Gamma^{-1}$. 
It holds $\delta(u)=\delta(v)$ if and only if $r(u)=r(v)$. 
Since the word $r(u)$ can be calculated from $u$ in linear time \cite{Book1982},
the word problem for $\FG(\Gamma)$ can be solved in linear time.
Let $\IRR(\Gamma) = \{ r(u) \mid u \in (\Gamma\cup\Gamma^{-1})^* \}$
be the set of all {\em irreducible} words.
The epimorphism $\delta :  (\Gamma\cup\Gamma^{-1})^* \to \FG(\Gamma)$
restricted to $\IRR(\Gamma)$ is a bijection.

The Cayley-graph of $\FG(\Gamma)$  with respect to the standard
generating set $\Gamma\cup\Gamma^{-1}$ will be denoted by $\Cmc(\Gamma)$.
Its vertex set is $\FG(\Gamma)$ and there is an $a$-labeled edge ($a \in \Gamma\cup\Gamma^{-1}$)
from $x\in\FG(\Gamma)$ to $y\in\FG(\Gamma)$  if $y = xa$ in $\FG(\Gamma)$.
Note that $\FG(\Gamma)$ 
is a finitely-branching tree. Figure~\ref{F1} shows a finite portion
of $\Cmc(\{a,b\})$. Here, and in the following, we 
only draw one directed edge between two points. Thus, for every drawn $a$-labeled 
edge we omit the $a^{-1}$-labeled reversed edge.
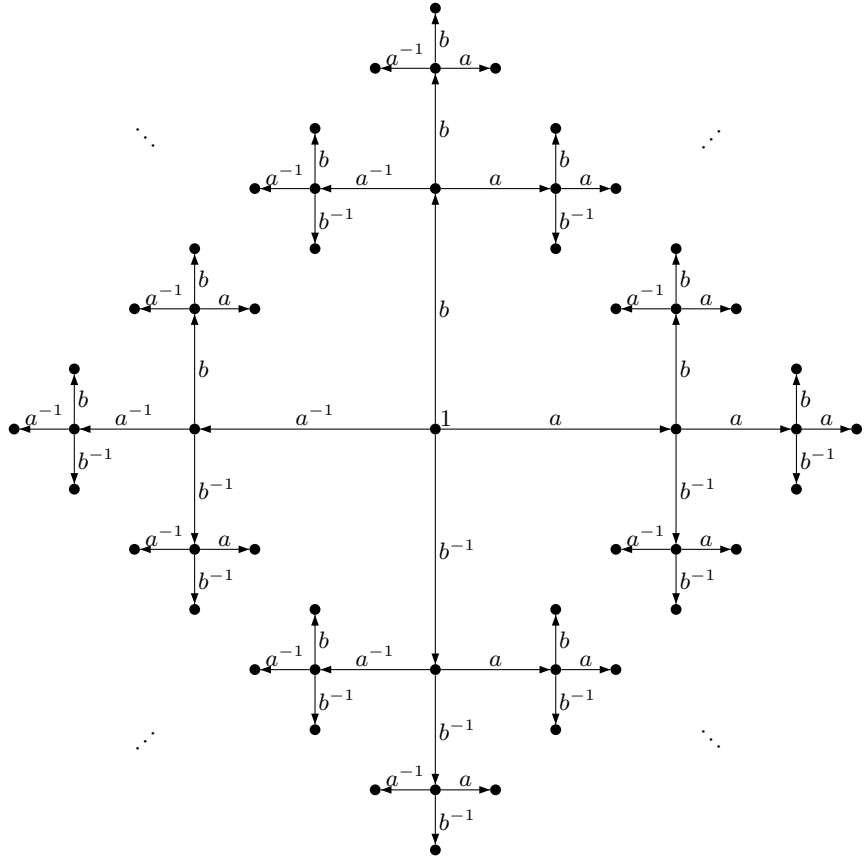
\begin{figure}[t]
  \begin{center} 
  \setlength{\unitlength}{1.6mm} 
  \begin{picture}(70,70)(5,5) 
    \gasset{Nframe=n,AHnb=1,ELdist=0.3} 
    \gasset{Nw=.9,Nh=.9,Nfill=y}

    \node[ExtNL=y,NLangle=45](1)(40,40){$1$}
    \put(15,65){\begin{turn}{-45}$\ldots$\end{turn}}
    \put(62,65){\begin{turn}{225}$\ldots$\end{turn}}
    \put(15,15){\begin{turn}{225}$\ldots$\end{turn}}
    \put(62,15){\begin{turn}{-45}$\ldots$\end{turn}}

    \node(-a)(20,40){}
    \node(a)(60,40){}
    \node(b)(40,60){}
    \node(-b)(40,20){}
    \drawedge(1,a){$a$}
    \drawedge[ELside=r](1,-a){$a^{-1}$}
    \drawedge(1,-b){$b^{-1}$}
    \drawedge[ELside=r](1,b){$b$}  

    \node(-aa)(10,40){}
    \node(-ab)(20,50){}
    \node(-a-b)(20,30){}
    \drawedge[ELside=r](-a,-aa){$a^{-1}$}
    \drawedge(-a,-a-b){$b^{-1}$}
    \drawedge[ELside=r](-a,-ab){$b$} 

    \node(aa)(70,40){}
    \node(ab)(60,50){}
    \node(a-b)(60,30){}
    \drawedge(a,aa){$a$}
    \drawedge(a,a-b){$b^{-1}$}
    \drawedge[ELside=r](a,ab){$b$}

    \node(b-a)(30,60){}
    \node(bb)(40,70){}
    \node(ba)(50,60){}
    \drawedge(b,ba){$a$}
    \drawedge[ELside=r](b,b-a){$a^{-1}$}
    \drawedge[ELside=r](b,bb){$b$}  

    \node(-b-a)(30,20){}
    \node(-bb)(40,10){}
    \node(-ba)(50,20){}
    \drawedge(-b,-ba){$a$}
    \drawedge[ELside=r](-b,-b-a){$a^{-1}$}
    \drawedge(-b,-bb){$b^{-1}$}

    \node(-aaa)(5,40){}
    \node(-aab)(10,45){}
    \node(-aa-b)(10,35){}
    \drawedge[ELside=r](-aa,-aaa){$a^{-1}$}
    \drawedge(-aa,-aa-b){$b^{-1}$}
    \drawedge[ELside=r](-aa,-aab){$b$} 

    \node(-ab-a)(15,50){}
    \node(-abb)(20,55){}
    \node(-aba)(25,50){}
    \drawedge(-ab,-aba){$a$}
    \drawedge[ELside=r](-ab,-ab-a){$a^{-1}$}
    \drawedge[ELside=r](-ab,-abb){$b$}  

    \node(-a-b-a)(15,30){}
    \node(-a-bb)(20,25){}
    \node(-a-ba)(25,30){}
    \drawedge(-a-b,-a-ba){$a$}
    \drawedge[ELside=r](-a-b,-a-b-a){$a^{-1}$}
    \drawedge(-a-b,-a-bb){$b^{-1}$}

    \node(b-aa)(25,60){}
    \node(b-ab)(30,65){}
    \node(b-a-b)(30,55){}
    \drawedge[ELside=r](b-a,b-aa){$a^{-1}$}
    \drawedge(b-a,b-a-b){$b^{-1}$}
    \drawedge[ELside=r](b-a,b-ab){$b$} 

    \node(baa)(55,60){}
    \node(bab)(50,65){}
    \node(ba-b)(50,55){}
    \drawedge(ba,baa){$a$}
    \drawedge(ba,ba-b){$b^{-1}$}
    \drawedge[ELside=r](ba,bab){$b$}

    \node(bb-a)(35,70){}
    \node(bbb)(40,75){}
    \node(bba)(45,70){}
    \drawedge(bb,bba){$a$}
    \drawedge[ELside=r](bb,bb-a){$a^{-1}$}
    \drawedge[ELside=r](bb,bbb){$b$}

    \node(aaa)(75,40){}
    \node(aab)(70,45){}
    \node(aa-b)(70,35){}
    \drawedge(aa,aaa){$a$}
    \drawedge(aa,aa-b){$b^{-1}$}
    \drawedge[ELside=r](aa,aab){$b$}  

    \node(ab-a)(55,50){}
    \node(abb)(60,55){}
    \node(aba)(65,50){}
    \drawedge(ab,aba){$a$}
    \drawedge[ELside=r](ab,ab-a){$a^{-1}$}
    \drawedge[ELside=r](ab,abb){$b$}  

    \node(a-b-a)(55,30){}
    \node(a-bb)(60,25){}
    \node(a-ba)(65,30){}
    \drawedge(a-b,a-ba){$a$}
    \drawedge[ELside=r](a-b,a-b-a){$a^{-1}$}
    \drawedge(a-b,a-bb){$b^{-1}$}

    \node(-b-aa)(25,20){}
    \node(-b-ab)(30,25){}
    \node(-b-a-b)(30,15){}
    \drawedge[ELside=r](-b-a,-b-aa){$a^{-1}$}
    \drawedge(-b-a,-b-a-b){$b^{-1}$}
    \drawedge[ELside=r](-b-a,-b-ab){$b$} 

    \node(-baa)(55,20){}
    \node(-bab)(50,25){}
    \node(-ba-b)(50,15){}
    \drawedge(-ba,-baa){$a$}
    \drawedge(-ba,-ba-b){$b^{-1}$}
    \drawedge[ELside=r](-ba,-bab){$b$}  

    \node(-bb-a)(35,10){}
    \node(-bbb)(40,5){}
    \node(-bba)(45,10){}
    \drawedge(-bb,-bba){$a$}
    \drawedge[ELside=r](-bb,-bb-a){$a^{-1}$}
    \drawedge(-bb,-bbb){$b^{-1}$}
  \end{picture}
  \end{center}
  \caption{The Cayley-graph $\Cmc(\{a,b\})$ of the free group $\FG(\{a,b\})$}
  \label{F1}
\end{figure}

\section{Inverse monoids}\label{S inverse}

A monoid $M$ is called an \emph{inverse monoid} 
if for every $m \in M$ there is a
{\em unique} $m^{-1} \in M$ such
that $m = mm^{-1}m$ and $m^{-1} = m^{-1}mm^{-1}$. 
For detailed reference on inverse monoids see 
\cite{Law99}; here we only recall the basic notions.
Since the class of inverse monoids forms a variety of algebras 
(with respect to the operations of multiplication, inversion, and the identity
element), the free inverse monoid $\FIM(\Gamma)$ 
generated by a set $\Gamma$ exists.
Vagner gave an explicit presentation of $\FIM(\Gamma)$: 
Let $\rho$ be the smallest congruence on the 
free monoid $(\Gamma \cup \Gamma^{-1})^*$ which
contains for all words $v,w \in (\Gamma \cup \Gamma^{-1})^*$ the pairs 
$(w, ww^{-1}w)$ and $(ww^{-1}vv^{-1},vv^{-1}ww^{-1})$; these
identities are also called Vagner equations. Then
$\FIM(\Gamma) \simeq (\Gamma \cup \Gamma^{-1})^*/\rho$.
An element $x$ of an inverse monoid $M$ is idempotent (i.e., $x^2=x$) if and only
if $x$ is of the form $mm^{-1}$ for some $m \in M$. Hence, Vagner's
presentation of $\FIM(\Gamma)$ implies that idempotent elements in an inverse monoid commute. 
Since the Vagner equations are true in the free group $\FG(\Gamma)$, 
there exists a congruence $\gamma$ on $\FIM(\Gamma)$ such that
$\FG(\Gamma) = \FIM(\Gamma)/\gamma$. When viewing congruences
as homomorphisms, we have $\delta = \rho \circ \gamma$, where
$\delta$ is the congruence on $(\Gamma \cup \Gamma^{-1})^*$ 
from \eqref{def-free-group}.

Elements of  $\FIM(\Gamma)$ can be also represented via \emph{Munn trees}:
The Munn tree $\MT(u)$ of $u \in (\Gamma \cup \Gamma^{-1})^*$ is a finite and prefix-closed
subset of $\IRR(\Gamma)$; it is defined by
\[  \MT(u) = \{  r(v) \mid v \preceq u  \}. \]
By identifying an irreducible word $v \in \IRR(\Gamma)$ with the group element
$\delta(v)$, $\MT(u)$ becomes the set of all nodes along the unique path in $C(\Gamma)$ 
that starts in $1$ and that is labeled with the word $u$. 
The subgraph of the Cayley-graph $C(\Gamma)$, which is induced 
by $\MT(u)$ is connected. Hence it is a finite tree and we can identify
$\MT(u)$ with this tree. The following result is known as Munn's Theorem:

\begin{theorem}[\protect{\cite{munn:1974}}] \label{munn}
For all $u,v \in (\Gamma \cup \Gamma^{-1})^*$, we have:
$\rho(u) = \rho(v)$ if and only if 
($r(u) = r(v)$ and $\MT(u) = \MT(v))$.
\end{theorem}
Thus, $\rho(u)\in\FIM(\Gamma)$ can be uniquely represented by the pair
$(\MT(u), r(u))$.
In fact, if we define 
on the set of all pairs $(U,v) \in 2^{\IRR(\Gamma)} \times \IRR(\Gamma)$
(with $v \in U$ and $U$ finite and prefix-closed)
a multiplication by
$(U,v)(V,w) = (r(U \cup vV), r(vw))$, 
then the resulting monoid is isomorphic to $\FIM(\Gamma)$.

Quite often, we will represent an element $\rho(u) \in \FIM(\Gamma)$ 
by a diagram for its Munn tree, where in addition 
 the node $\varepsilon $ is represented by a bigger circle and the node $r(u)$ is marked
by an outgoing arrow. If $r(u) = \varepsilon$, then we omit this arrow. By Thm.~\ref{munn}
such a diagram uniquely specifies an element of $\FIM(\Gamma)$.

\begin{example}
The diagram for $\rho(b b^{-1} a b b^{-1} a) \in \FIM(\{a,b\})$ looks
as follows:
\begin{center}
  \setlength{\unitlength}{1mm}
  \begin{picture}(20,14)(0,-2)
    \gasset{Nadjust=n,ExtNL=y,NLdist=1,Nw=1,Nh=1,AHnb=1,Nfill=y,ELdist=0.5,AHLength=2,AHlength=2,linewidth=.3,Nfill=y}
    \node(b)(0,10){}
    \node[Nw=2,Nh=2,Nfill=n](0)(0,0){}
    \node(a)(10,0){}
    \node(ab)(10,10){}
    \node(aa)(20,0){}
    \drawline(20,0)(23.5,0)
    \drawedge(a,ab){$b$}
    \drawedge(0,a){$a$}
    \drawedge(a,aa){$a$}
    \drawedge(0,b){$b$}
 \end{picture}
\end{center}
\end{example}
Thm.~\ref{munn} leads to a polynomial time algorithm for the word problem for
$\FIM(\Gamma)$. For instance, the reader can easily check that 
$b b^{-1} a b b^{-1} a = a a a^{-1} b b^{-1} a^{-1} b b^{-1} a a$ 
in $\FIM(\{a,b\})$ by using Munn's Theorem. In fact, every word that labels
a path from $\varepsilon$ to $aa$  (the node with the outgoing arrow) and that visits
all nodes of the above diagram represents the same element of $\FIM(\{a,b\})$
as $b b^{-1} a b b^{-1} a$.
Munn's theorem also implies that an element 
$\rho(u)\in\FIM(\Gamma)$ (where 
$u\in (\Gamma \cup \Gamma^{-1})^*$) 
is idempotent (i.e., $\rho(uu) = \rho(u)$) if and only if $r(u) = \varepsilon$.

For a finite set $P \subseteq (\Gamma \cup \Gamma^{-1})^* \times (\Gamma \cup \Gamma^{-1})^*$
define $\FIM(\Gamma)/P = (\Gamma \cup \Gamma^{-1})^*/\tau_P$ to be the inverse
monoid with the set $\Gamma$ of generators and the set $P$
of relations, where $\tau_P$ is the smallest congruence on $(\Gamma \cup \Gamma^{-1})^*$
generated by  $\rho \cup P$. Viewed as a morphism, this
congruence 
factors as $\tau_P = \rho \circ  \nu_P$ with
$\FIM(\Gamma)/\nu_P = \FIM(\Gamma)/P$.
We say that $P \subseteq (\Gamma \cup \Gamma^{-1})^* \times (\Gamma
\cup \Gamma^{-1})^*$ is an \emph{idempotent presentation} if for all $(e,f) \in P$, 
$\rho(e)$ and $\rho(f)$ are both idempotents of $\FIM(\Gamma)$, i.e.,
$r(e)=r(f)=\varepsilon$ by the remark above. 
In this paper, we are concerned with inverse monoids of the form 
$\FIM(\Gamma)/P$ for a finite idempotent presentation $P$. 
In this case, since every identity $(e,f) \in P$ is true in $\FG(\Gamma)$
(we have $\delta(e)=\delta(f)=1$), there also exists a 
congruence $\gamma_P$ on $\FIM(\Gamma)/P$ with 
$(\FIM(\Gamma)/P)/\gamma_P = \FG(\Gamma)$.
The following commutative diagram summarizes all morphisms introduced
so far. 
\begin{center}
\setlength{\unitlength}{1.5mm}
\begin{picture}(40,15)(0,-5)
\gasset{Nadjust=wh,Nadjustdist=.5,Nframe=n,Nfill=n,ELdist=0.5,AHnb=1,AHLength=1.5,AHlength=1.5,linewidth=.2}
\node(FM)(20,10){$(\Gamma \cup \Gamma^{-1})^*$}
\node(FIM)(0,0){$\FIM(\Gamma)$}
\node(FG)(40,0){$\FG(\Gamma)$}
\node(Inv)(20,0){$\FIM(\Gamma)/P$}
\drawedge[ELside=r](FM,FIM){$\rho$}
\drawedge[curvedepth=-3,ELside=r](FIM,FG){$\gamma$}
\drawedge(FM,FG){$\delta$}
\drawedge[ELside=r](FM,Inv){$\tau_P$}
\drawedge[ELpos=50](FIM,Inv){$\nu_P$}
\drawedge(Inv,FG){$\gamma_P$}
\end{picture}
\end{center}
In the sequel, the meaning of the congruences
$\rho,\delta,\gamma_P,\gamma,\tau_P$, and $\nu_P$ will be fixed.

To solve the word problem for $\FIM(\Gamma)/P$, Margolis and Meakin
\cite{margolis:1993} used a closure operation for Munn trees, which
is based on work of Stephen \cite{stephen:1990}. We shortly
review the ideas here. As remarked in \cite{margolis:1993}, every idempotent
presentation $P$ can be replaced by the idempotent presentation $P' = \{ (e,ef), (f,ef) \mid (e,f)
\in P\}$, i.e., $\FIM(\Gamma)/P = \FIM(\Gamma)/P'$. 
Since $\MT(e) \subseteq \MT(ef) \supseteq \MT(f)$ if $r(e)=r(f)=\varepsilon$,
we can restrict in the following to idempotent presentations $P$
such that $\MT(e) \subseteq \MT(f)$ for all $(e,f) \in P$.  
Define a rewriting relation $\Rightarrow_P$ on prefix-closed subsets 
of $\IRR(\Gamma)$ as follows, where $U,V \subseteq \IRR(\Gamma)$:
$U \Rightarrow_P V$ if and only if 
$$
\exists (e,f) \in P \; \exists u \in U \big(
 r(u\, \MT(e)) \subseteq  U \; \text{and} \;
V = U \cup r(u\, \MT(f)) \big).
$$
Finally, define the closure of $U \subseteq \IRR(\Gamma)$ with respect to the presentation $P$
as 
$$\cl_P(U) = \bigcup \{ V \mid U \stackrel{*}{\Rightarrow}_P V \}.$$
\begin{example} \label{E 1}
Assume that $\Gamma = \{a,b\}$, $P = \{ (a a^{-1}, a^2 a^{-2}), (b b^{-1}, b^2
b^{-2})\}$ and $u = aa^{-1}bb^{-1}$.
The graphical representations for these elements look as follows:
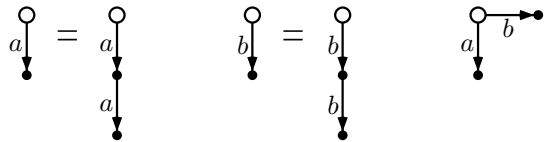
\begin{center}
\setlength{\unitlength}{1mm}
\begin{picture}(68,20)(0,-2)
\gasset{Nadjust=n,ExtNL=y,NLdist=1,Nw=1,Nh=1,AHnb=1,Nfill=y,ELdist=0.5,AHLength=2,AHlength=2,linewidth=.3}
\node[Nw=2,Nh=2,Nfill=n](a)(0,16){}
\node(b)(0,8){}
\drawedge[ELside=r](a,b){$a$}
\node[Nw=2,Nh=2,Nfill=n](c)(12,16){}
\node(d)(12,8){}
\node(e)(12,0){}
\drawedge[ELside=r](c,d){$a$}
\drawedge[ELside=r](d,e){$a$}
\put(4,12){\Large =}

\node[Nw=2,Nh=2,Nfill=n](a')(30,16){}
\node(b')(30,8){}
\drawedge[ELside=r](a',b'){$b$}
\node[Nw=2,Nh=2,Nfill=n](c')(42,16){}
\node(d')(42,8){}
\node(e')(42,0){}
\drawedge[ELside=r](c',d'){$b$}
\drawedge[ELside=r](d',e'){$b$}
\put(34,12){\Large =}

\node[Nw=2,Nh=2,Nfill=n](a'')(60,16){}
\node(b'')(60,8){}
\drawedge[ELside=r](a'',b''){$a$}
\node(c'')(68,16){}
\drawedge[ELside=r](a'',c''){$b$}
\end{picture}
\end{center} 
Then the closure $\cl_P(\MT(u))$ is $\{ a^n \mid n \geq 0\} \cup 
\{ b^n \mid n \geq 0\} \subseteq \IRR(\Gamma)$. 
\end{example}
\noindent
Margolis and Meakin proved the following result:

\begin{theorem}[\cite{margolis:1993}]\label{wordproblem}
Let $P$ be an idempotent presentation and let $u,v \in (\Gamma \cup \Gamma^{-1})^*$. 
Then $\tau_P(u) = \tau_P(v)$ if and only if  $(r(u)=r(v)$
and $\cl_P(\MT(u)) =  \cl_P(\MT(v)))$.
\end{theorem}
\noindent
The result of Munn for $\FIM(\Gamma)$ (Thm.~\ref{munn})
is a special case of this result for $P = \emptyset$.
Note also that $\cl_P(\MT(u)) =  \cl_P(\MT(v))$ if and only
if $\MT(u) \subseteq  \cl_P(\MT(v))$ and $\MT(v) \subseteq  \cl_P(\MT(u))$.
Margolis and Meakin used Thm.~\ref{wordproblem} in connection
with \mbox{Rabin's} tree theorem
in order to give a solution for 
the word problem for the monoid $\FIM(\Gamma)/P$. 
Using tree automata techniques, a logspace algorithm
for the word problem for  $\FIM(\Gamma)/P$ was given in \cite{LoOn06IC}.
For this result, it is important that the idempotent presentation $P$ is 
not part of the input. The uniform version of the word problem,
where $P$ is part of the input, is $\EXPTIME$-complete  \cite{LoOn06IC}.

\section{Straight-line programs} \label{S SLP}

We are using straight-line programs as a succinct representation of
strings with reoccurring subpatterns \cite{PlRy99}. A
\textit{straight-line program (SLP) over a finite alphabet $\Gamma$} is a
context free grammar $\dA=(V,\Gamma,S,P)$, where $V$ is the set of
\textit{nonterminals}, $\Gamma$ is the set of \textit{terminals},
$S\in V$ is the \textit{initial nonterminal}, and $P\subseteq V \times
(V\cup \Gamma)^*$ is the set of \textit{productions} such that (i) for
every $X\in V$ there is exactly one $\alpha \in (V \cup \Gamma)^*$
with $(X,\alpha)\in P$ and (ii) there is no cycle in the relation
$\{(X,Y)\in V\times V \mid  \exists \alpha \in  (V \cup \Gamma)^* Y  (V \cup \Gamma)^* : 
(X,\alpha) \in P \}$. These conditions ensure that the language
generated by the straight-line program $\dA$ contains exactly one word
$\eval(\dA)$.
The size of $\dA$ is $|\dA|= \sum_{(X,\alpha)\in P}|\alpha|$. 

\begin{remark}  \label{remark:SLP}
The following problems can be solved in  polynomial time:
\begin{enumerate}[(a)]
\item Given an SLP $\dA$, calculate $|\eval(\dA)|$ in binary representation.
\item Given an SLP $\dA$ and two binary coded numbers $1 \leq i \leq j \leq  |\eval(\dA)|$,
compute an SLP $\dB$ with $\eval(\dB) = \eval(\dA)[i,j]$.
\end{enumerate}
Also notice that $\eval(\dA)$ can be computed from $\dA$ by a $\PSPACE$-transducer.
\end{remark}
In \cite{Pla94}, Plandowski presented a polynomial time algorithm
for testing whether $\eval(\dA) = \eval(\dB)$ for two given 
SLPs $\dA$ and $\dB$. A cubic algorithm was presented by Lifshits \cite{Lif07}.
In fact, Lifshits gave an algorithm for compressed pattern
matching: given SLPs $\dA$ and $\dB$, is $\eval(\dA)$ a factor of $\eval(\dB)$?
The running time of his algorithm is $O(|\dA| \cdot |\dB|^2)$.

Let $M$ be a finitely generated monoid and let $\Gamma$ be a finite 
generating set for $M$. The {\em compressed word problem} 
for $M$ is the following computational problem:

\medskip
\noindent
INPUT: SLPs $\dA$ and $\dB$ over the alphabet $\Gamma$.

\noindent
QUESTION: Does $\eval(\dA) = \eval(\dB)$ hold in $M$?

\medskip
\noindent
The above mentioned result of  Plandowski \cite{Pla94} means that
the compressed word problem for a finitely generated free monoid
can be solved in polynomial time. 
The following result was shown in \cite{Loh06siam}.

\begin{theorem}[\protect{\cite{Loh06siam}}] \label{CWP-F}
For every finite alphabet $\Gamma$, 
the compressed word problem for $\FG(\Gamma)$ 
can be solved in polynomial time (and is $\Ptime$-complete if 
$|\Gamma| \geq 2$).
\end{theorem}

\section{Compressed word problem for $\FIM(\Gamma)$}

Recall that the word problem for $\FIM(\Gamma)$ can be solved in logspace \cite{LoOn06IC}.
In the compressed setting we have:

\begin{theorem} \label{thm:cwp-pi2p}
For every finite alphabet $\Gamma$ with $|\Gamma| \geq 2$, 
the compressed word problem for $\FIM(\Gamma)$ is $\Pi^p_2$-complete.
\end{theorem}

\begin{proof}
For the $\Pi^p_2$ upper bound, let $\dA$ and $\dB$ be 
SLPs over some alphabet $\Gamma \cup \Gamma^{-1}$ and let
$m = |\eval(\dA)|$ and $n = |\eval(\dB)|$. These numbers
can be computed in polynomial time by Remark~\ref{remark:SLP}.
By Thm.~\ref{munn}, we have
$\eval(\dA) = \eval(\dB)$ in $\FIM(\Gamma)$ if and only if:
\begin{align}
& \eval(\dA) = \eval(\dB) \text{ in } \FG(\Gamma) \label{eqFG1} \\
& \forall i \in \{0,\ldots, m\} \,\exists j \in \{0,\ldots, n\} : \eval(\dA)[1,i] =
\eval(\dB)[1,j] \text{ in } \FG(\Gamma) \label{eqFG2}\\
& \forall i \in \{0,\ldots, n\} \,\exists j \in \{0,\ldots, m\} : \eval(\dB)[1,i] =
\eval(\dA)[1,j] \text{ in } \FG(\Gamma) \label{eqFG3}
\end{align}
Thm.~\ref{CWP-F} implies that (\ref{eqFG1}) can be checked in 
polynomial time, whereas (\ref{eqFG2}) and (\ref{eqFG3}) are 
$\Pi^p_2$-properties.

It suffices to prove the lower bound for $\Gamma=\{a,b\}$.
We make a logspace reduction from the following $\Pi^p_2$-complete
problem \cite{BeKaLaPlRy02}, where
$\overline{u} \cdot \overline{v}
= u_1 v_1 + \cdots + u_n v_n$ denotes the scalar product of two
integer vectors $\overline{u} = (u_1,\ldots,u_n)$,
$\overline{v} = (v_1,\ldots,v_n)$:

\medskip

\noindent
INPUT: vectors $\overline{u}=(u_1,\ldots,u_m) \in \mathbb{N}^m$, 
$\overline{v}=(v_1,\ldots,v_n) \in \mathbb{N}^n$, and $t \in \mathbb{N}$ (all coded binary)

\noindent
QUESTION: Does $\forall \overline{x} \in \{0,1\}^m \exists \overline{y} \in
\{0,1\}^n : \overline{u} \cdot \overline{x} + \overline{v} \cdot
\overline{y} = t$ hold?

\medskip

\noindent
Let $s = u_1 + \cdots + u_m + v_1 + \cdots + v_n$,
$s_u = u_1 + \cdots + u_m$, and
$s_v =  v_1 + \cdots + v_n$.
W.l.o.g. we can assume $t < s$.
Using the construction from 
\cite{Loh06siam} (proof of Theorem 5.2) we can construct in logspace an SLP $\dA_1$ such that
$\eval(\dA_1) = \prod_{\overline{x} \in \{0,1\}^m}  a^{\overline{u} \cdot
  \overline{x}} A_1 a^{s_u - \overline{u} \cdot \overline{x}}$.
Here the product is taken over all tuples from $ \{0,1\}^m$ in lexicographic 
order.
By replacing $A_1$ by $A_2 a^{s_v}$ (which can be easily generated by a
small SLP), we obtain an SLP $\dA_2$ with 
$\eval(\dA_2) = \prod_{\overline{x} \in \{0,1\}^m}  a^{\overline{u} \cdot
  \overline{x}} A_2 a^{s - \overline{u} \cdot \overline{x}}$.
Similarly, we obtain an SLP $\dA_3$ with
$\eval(\dA_3) = \prod_{\overline{y} \in \{0,1\}^n}  a^{\overline{v} \cdot
  \overline{y}} (b b^{-1} a^{-s_v}) a^{s_v - \overline{v} \cdot \overline{y}}$.
Finally, be replacing $A_2$ in $\dA_2$ by the start nonterminal of $\dA_3$
we obtain an SLP $\dA$ with
$$
\eval(\dA) = \prod_{\overline{x} \in \{0,1\}^m}  \biggl[
a^{\overline{u} \cdot \overline{x}} 
\prod_{\overline{y} \in \{0,1\}^n}  \biggl( a^{\overline{v} \cdot
  \overline{y}} b b^{-1} a^{-s_v} a^{s_v - \overline{v} \cdot \overline{y}} \biggr)
a^{s - \overline{u} \cdot \overline{x}}
\biggr].
$$
Moreover, it is easy to construct a second SLP $\dB$ such that
$$
\eval(\dB) = \eval(\dA) a^{-s \cdot 2^m} \bigl( a^t b b^{-1} a^{s-t} \bigr)^{2^m} .
$$
We claim that $\eval(\dA) = \eval(\dB)$ in $\FIM(\{a,b\})$ if and only if
\begin{equation} \label{subsetsum}
\forall \overline{x} \in \mathbb{N}^m \exists \overline{y} \in
\mathbb{N}^n : \overline{u} \cdot \overline{x} + \overline{v} \cdot
\overline{y} = t.
\end{equation}
We have $r(\eval(\dA)) = r(\eval(\dB)) = a^{s \cdot 2^m}$. Thus,
$\eval(\dA) = \eval(\dB)$ holds in $\FIM(\{a,b\})$ 
if and only if $\MT(\eval(\dA)) = \MT(\eval(\dB))$.
Since $\eval(\dA)$ is a prefix of $\eval(\dB)$,
we obtain $\MT(\eval(\dA)) \subseteq \MT(\eval(\dB))$.
Moreover, for the prefix $\eval(\dA) a^{-s \cdot 2^m}$ of $\eval(\dB)$
we have $r(\eval(\dA) a^{-s \cdot 2^m}) = \varepsilon$ and
$\MT(\eval(\dA) a^{-s \cdot 2^m}) = \MT(\eval(\dA))$.
This and the fact that $\MT(\eval(\dA)) \subseteq \MT(\eval(\dB))$ 
implies that $\MT(\eval(\dA)) = \MT(\eval(\dB))$ if and only if 
\begin{equation}\label{ink3}
\MT( (a^t b b^{-1} a^{s-t})^{2^m} ) \subseteq \MT(\eval(\dA)) .
\end{equation}
We show that (\ref{ink3}) is equivalent to (\ref{subsetsum}).
We have
$$
\MT( (a^t b b^{-1} a^{s-t})^{2^m} ) = \{ a^i \mid 0 \leq i \leq s \cdot 2^m\}
\cup \{ a^{t + k \cdot s} b \mid 0 \leq k < 2^m \}.  
$$
Since $r(\eval(\dA)) = a^{s \cdot 2^m}$,
we have $a^i \in \MT(\eval(\dA))$ for all $0 \leq i \leq s \cdot 2^m$.
Hence, (\ref{ink3}) is equivalent to
$a^{t + k \cdot s} b \in \MT(\eval(\dA))$
for every $0 \leq k < 2^m$, i.e.
(for a bit vector $\overline{u} = (u_1,\ldots,u_n) \in \{0,1\}^n$
let $n(\overline{u}) = \sum_{i=1}^n u_i 2^{i-1}$ be the number 
represented by $\overline{u}$)
\begin{equation} \label{memb1}
 \forall \overline{x} \in \{0,1\}^m : a^{n(\overline{x}) \cdot s + t} b \in \MT(\eval(\dA)) .
\end{equation}
Now, 
$\MT(\eval(\dA)) \cap a^* b = \{ a^{n(\overline{x}) \cdot s + \overline{u} \cdot \overline{x} + \overline{v} \cdot
\overline{y}} b \mid \overline{x} \in \{0,1\}^m, \overline{y} \in \{0,1\}^n \}$.
Hence, (\ref{memb1}) if and only if
$\forall \overline{x} \in \{0,1\}^m \exists \overline{y} \in \{0,1\}^n : \overline{u} \cdot \overline{x} + \overline{v} \cdot
\overline{y} = t$.
This concludes the proof.
\qed
\end{proof}
For a free inverse monoid of rank one, the compressed word problem is
simpler:

\begin{proposition} \label{prop:rank-one}
The compressed word problem for $\FIM(\{a\})$ can be
solved in polynomial time.
\end{proposition}

\begin{proof}
Note that the free group $\FG(\{a\})$ is isomorphic to $\mathbb{Z}$.
An element of $\FIM(\{a\})$ can be represented by a triple 
$(i,j,k) \in \mathbb{Z}^3$; where $i \leq j \leq k$, $i \leq 0 \leq k$.
This triple represents the element $x \in \FIM(\{a\})$, where
$\gamma(x) = j$ and the Munn tree
is $\{i, \ldots, k\} \subseteq \mathbb{Z}$.
Multiplication of these triples is defined as
$$
(i_1,j_1,k_1) \cdot (i_2,j_2,k_2) = (\min\{i_1, j_1+i_2\}, j_1+j_2, \max\{k_1,j_1+k_2\}) .
$$
From this rule, it is easy to compute in polynomial time for every variable $A$ of an SLP $\dA$
the $\mathbb{Z}$-triple that represents $\rho(\eval(A))$.
\qed
\end{proof}

\section{Compressed word problems for $\FIM(\Gamma)/P$}

For an inverse monoid of the form $\FIM(\Gamma)/P$,
where $\Gamma$ is finite and $P$ is a finite idempotent presentation, the 
word problem can be still solved in logspace \cite{LoOn06IC}.
In this case, the complexity of the compressed word problem reaches 
even $\PSPACE$:

\begin{theorem} \label{thm:PSPACE-idempot}
The following holds:
\begin{enumerate}[(a)]
\item For every  finite
idempotent presentation $P \subseteq (\Gamma \cup \Gamma^{-1})^* 
\times (\Gamma \cup \Gamma^{-1})^*$, the compressed word
problem for $\FIM(\Gamma)/P$ belongs to $\PSPACE$.
\item There exists a fixed finite
idempotent presentation $P \subseteq (\Gamma \cup \Gamma^{-1})^* 
\times (\Gamma \cup \Gamma^{-1})^*$ such that the compressed word
problem for $\FIM(\Gamma)/P$ is $\PSPACE$-complete.
\end{enumerate}
\end{theorem}

\begin{proof}
Let us first show (a).
In \cite{LoOn06IC}, it was shown that the ordinary word problem for $\FIM(\Gamma)/P$
can be solved in logarithmic space. Since $\eval(\dA)$ can be computed from 
$\dA$ by a $\PSPACE$-transducer (Remark~\ref{remark:SLP}), statement
(a) follows from Lemma~\ref{PSPACE}.

For the lower bound in (b), we use the following recent result from \cite{Loh11IC}: There
exists a fixed regular language $L$ over some paired alphabet $\Sigma \times
\Theta$ such that the following problem is
$\PSPACE$-complete (for strings $u \in \Sigma^*, v \in \Theta^*$ with $|u| = |v| =n$
let $u \otimes v = (u[1],v[1]) \cdots (u[n],v[n]) \in (\Sigma \times \Theta)^*$):

\medskip

\noindent
INPUT: SLPs $\dA$ (over $\Sigma$) and $\dB$ (over $\Theta$)
with $|\eval(\dA)| = |\eval(\dB)|$

\noindent
QUESTION: Does $\eval(\dA) \otimes \eval(\dB) \in L$ hold?

\medskip

\noindent
W.l.o.g. assume that $\Sigma \cap \Theta = \emptyset$.
Let $\mcA = (Q, \Sigma \times \Theta, \delta, q_0, F)$ be a
deterministic finite automaton with $L(\mcA)=L$.
Let $\Gamma = \Sigma \cup \Theta \cup Q \cup \{A,B,C\}$
(all unions are assumed to be disjoint).
Consider the fixed idempotent presentation over the alphabet $\Gamma$ with the following
relations:
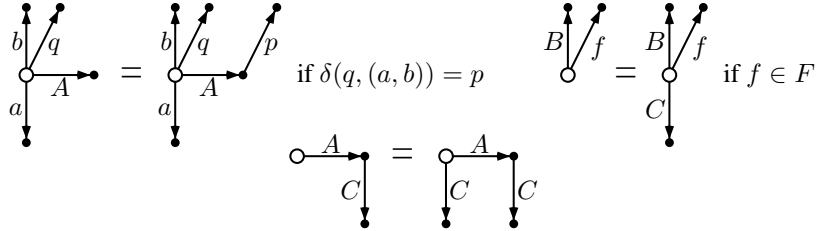
\begin{center}
\setlength{\unitlength}{.9mm}
\begin{picture}(115,34)(0,18)
\gasset{Nadjust=n,ExtNL=y,NLdist=1,Nw=1,Nh=1,AHnb=1,Nfill=y,ELdist=0.5,AHLength=2,AHlength=2,linewidth=.3}
\node[Nw=2,Nh=2,Nfill=n](0)(0,40){}
\node(a)(0,30){}
\node(b)(0,50){}
\node(q)(5,50){}
\node(A)(10,40){}
\drawedge[ELside=r](0,a){$a$}
\drawedge(0,b){$b$}
\drawedge[ELside=r](0,q){$q$}
\drawedge[ELside=r](0,A){$A$}
\put(14,39){\Large =}
\node[Nw=2,Nh=2,Nfill=n](0')(22,40){}
\node(a')(22,30){}
\node(b')(22,50){}
\node(q')(27,50){}
\node(A')(32,40){}
\node(p)(37,50){}
\drawedge[ELside=r](0',a'){$a$}
\drawedge(0',b'){$b$}
\drawedge[ELside=r](0',q'){$q$}
\drawedge[ELside=r](0',A'){$A$}
\drawedge[ELside=r](A',p){$p$}
\put(40,39){if $\delta(q, (a,b)) = p$}

\node[Nw=2,Nh=2,Nfill=n](0'')(80,40){}
\node(B)(80,50){}
\node(f)(85,50){}
\drawedge(0'',B){$B$}
\drawedge[ELside=r](0'',f){$f$}
\put(87,39){\Large =}
\node[Nw=2,Nh=2,Nfill=n](0''')(95,40){}
\node(B')(95,50){}
\node(C)(95,30){}
\node(f')(100,50){}
\drawedge(0''',B'){$B$}
\drawedge[ELside=r](0''',C){$C$}
\drawedge[ELside=r](0''',f'){$f$}
\put(103,39){if $f \in F$}

\node[Nw=2,Nh=2,Nfill=n](1)(40,28){}
\node(A'')(50,28){}
\node(C')(50,18){}
\put(54,27){\Large =}
\node[Nw=2,Nh=2,Nfill=n](2)(62,28){}
\node(A''')(72,28){}
\node(C'')(72,18){}
\node(C''')(62,18){}
\drawedge(1,A''){$A$}
\drawedge[ELside=r](A'',C'){$C$}
\drawedge(2,A'''){$A$}
\drawedge(A''',C''){$C$}
\drawedge(2,C'''){$C$}
\end{picture}
\end{center} 
With the upper left relation, we simulate the automaton $\mcA$. The upper right 
relation allows to add a $C$-labeled edge as soon as a final state is reached;
the $B$-labeled edge acts as a kind of end marker for the input word. 
Finally, the last relation allows to propagate the $C$-labeled edge back to the 
origin (node $1$).

Assume that $\eval(\dA) = a_1 \cdots a_n$ and 
$\eval(\dB) = b_1 \cdots b_n$.
Consider the string 
$$
w = q_0 q_0^{-1} \prod_{i=1}^n (a_i a_i^{-1} A) BB^{-1} \prod_{i=0}^{n-1} (A^{-1} b_{n-i} b_{n-i}^{-1}). 
$$
It is easy to compute from $\dA$ and $\dB$ in polynomial time an SLP 
$\dC$ with $\eval(\dC) = w$.
The Munn tree $\MT(w)$ looks as follows:
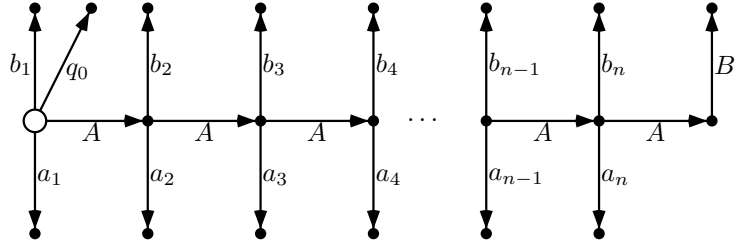
\begin{center}
\setlength{\unitlength}{1.5mm}
\begin{picture}(60,20)(0,0)
\gasset{Nadjust=n,ExtNL=y,NLdist=1,Nw=.8,Nh=.8,AHnb=1,Nfill=y,ELdist=0.2,AHLength=1.7,AHlength=1.7,linewidth=.2}
\node[Nw=2,Nh=2,Nfill=n](0)(0,10){}
\node(a1)(0,0){}
\node(b1)(0,20){}
\node(q)(5,20){}
\node(1)(10,10){}
\node(a2)(10,0){}
\node(b2)(10,20){}
\node(2)(20,10){}
\node(a3)(20,0){}
\node(b3)(20,20){}
\node(3)(30,10){}
\node(a4)(30,0){}
\node(b4)(30,20){}
\put(33,10){$\ldots$}
\node(n-1)(40,10){}
\node(an-1)(40,0){}
\node(bn-1)(40,20){}
\node(n)(50,10){}
\node(an)(50,0){}
\node(bn)(50,20){}
\node(n+1)(60,10){}
\node(bn+1)(60,20){}
\drawedge[ELside=r](0,1){$A$}
\drawedge[ELside=r](1,2){$A$}
\drawedge[ELside=r](2,3){$A$}
\drawedge[ELside=r](n-1,n){$A$}
\drawedge[ELside=r](n,n+1){$A$}
\drawedge(0,a1){$a_1$}
\drawedge(0,b1){$b_1$}
\drawedge[ELside=r](0,q){$q_0$}
\drawedge(1,a2){$a_2$}
\drawedge[ELside=r](1,b2){$b_2$}
\drawedge(2,a3){$a_3$}
\drawedge[ELside=r](2,b3){$b_3$}
\drawedge(3,a4){$a_4$}
\drawedge[ELside=r](3,b4){$b_4$}
\drawedge(n-1,an-1){$a_{n-1}$}
\drawedge[ELside=r](n-1,bn-1){$b_{n-1}$}
\drawedge(n,an){$a_n$}
\drawedge[ELside=r](n,bn){$b_n$}
\drawedge[ELside=r](n+1,bn+1){$B$}
\end{picture}
\end{center}
We claim that $w = CC^{-1} w$ in $\FIM(\Gamma)/P$ if and only if 
$\eval(\dA) \otimes \eval(\dB) \in L(\mcA)$.
Clearly, $w = CC^{-1} w = 1$ in $\FG(\Gamma)$.
Moreover, $\cl_P(\MT(w)) = \cl_P(\MT(CC^{-1} w))$ if and only if 
$C \in \cl_P(\MT(w))$. Thus, it suffices to show that 
$C \in \cl_P(\MT(w))$ if and only if $\eval(\dA) \otimes \eval(\dB) \in L(\mcA)$.
First, assume that $\eval(\dA) \otimes \eval(\dB) \notin L(\mcA)$.
Let $q_i$ be the state of $\mcA$ after reading $(a_1,b_1) \cdots (a_i,b_i)$
($0 \leq i \leq n$). Thus, $q_n \not\in F$.
This implies that $\cl_P(\MT(w)) = \MT(w) \cup \{ A^i q_i \mid 0 \leq i \leq n\}$.
Hence, $C \not\in \cl_P(\MT(w))$. On the other hand, if $q_n \in F$, then
$\cl_P(\MT(w)) = \MT(w) \cup \{ A^i q_i, A^i C \mid 0 \leq i \leq n\}$
and therefore $C \in \cl_P(\MT(w))$.
\qed
\end{proof}

\section{Rational subset membership problems}

In this section we briefly outline our results 
on the compressed variant of the rational
subset membership problem for free inverse monoids.
We start with a lower bound.
Note that for $K \subseteq (\Gamma \cup \Gamma^{-1})^*$,
$\rho(K^*)$ is the submonoid of $\FIM(\Gamma)$ 
generated by $\rho(K)$. 

\begin{theorem} \label{PSPACE-hard submonoid}
There exists a fixed alphabet $\Gamma$ and a fixed
finite subset $K \subseteq (\Gamma \cup \Gamma^{-1})^*$
such that the following problem is $\PSPACE$-hard:

\noindent
INPUT: An SLP $\dA$ over the alphabet $\Gamma \cup \Gamma^{-1}$

\noindent
QUESTION: Does $\rho(\eval(\dA)) \in \rho(K^*)$ hold?
\end{theorem}

\begin{proof}
The proof is very similar to the proof of statement (b) from
Thm.~\ref{thm:PSPACE-idempot}. In fact, we use a reduction
from the same $\PSPACE$-complete problem that we used there.
So, take two SLPs $\dA$ (over $\Sigma$) and $\dB$ (over $\Theta$)
with $|\eval(\dA)| = |\eval(\dB)|$. Again, let 
$\mcA = (Q, \Sigma \times \Theta, \delta, q_0, F)$ be a
deterministic finite automaton with $L(\mcA)=L$.
Let $\Gamma = \Sigma \cup \Theta \cup Q \cup \{A,B\}$
(all unions are assumed to be disjoint).
W.l.o.g. we can assume that final states of $\mcA$ do not have
outgoing transitions (to ensure this, one can introduce a copy
for each final state).
We choose $K = K_1 \cup K_2$, where:
\begin{eqnarray*}
K_1   & =  & \{ aa^{-1}bb^{-1} qq^{-1}A \!\!\! \prod_{s \in Q \setminus \{p\}}   \!\!\!\! ss^{-1}\mid a \in \Sigma, b \in \Theta, q,p \in Q \setminus F, p=\delta((a,b),q) \}   \\
K_2   & =  &   \{ aa^{-1}bb^{-1} qq^{-1}A B B^{-1}\mid a \in \Sigma, b \in \Theta, q \in Q \setminus F, p=\delta((a,b),q) \in F \}.
\end{eqnarray*}
The Munn trees for the elements in $\rho(K)$ look as follows
(note that these elements are not idempotent).
In the following, edges labeled with a subset
$P \subseteq Q$ represents $|P|$ many edges labeled with the symbols
of $P$, where all these edges have the same origin but the target nodes are different.
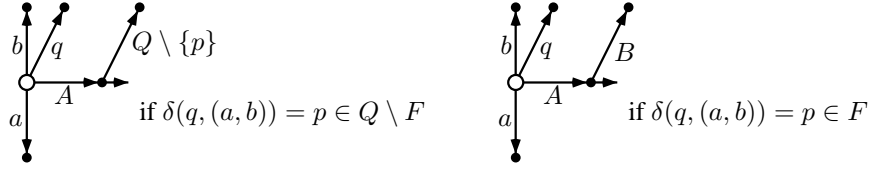
\begin{center}
\setlength{\unitlength}{1mm}
\begin{picture}(110,25)(0,27)
\gasset{Nadjust=n,ExtNL=y,NLdist=1,Nw=1,Nh=1,AHnb=1,Nfill=y,ELdist=0.5,AHLength=2,AHlength=2,linewidth=.3}
\node[Nw=2,Nh=2,Nfill=n](0)(0,40){}
\node(a)(0,30){}
\node(b)(0,50){}
\node(B)(5,50){}
\node(A)(10,40){}
 \drawline(10,40)(13.5,40)
\node(q)(15,50){}
\drawedge[ELside=r](0,a){$a$}
\drawedge(0,b){$b$}
\drawedge[ELside=r](0,B){$q$}
\drawedge[ELside=r](0,A){$A$}
\drawedge[ELside=r,ELpos=80](A,q){$Q \setminus \{p\}$}
\put(15,35){if $\delta(q, (a,b)) = p \in Q \setminus F$}
\node[Nw=2,Nh=2,Nfill=n](0')(65,40){}
\node(a')(65,30){}
\node(b')(65,50){}
\node(B')(70,50){}
\node(A')(75,40){}
\drawline(75,40)(78.5,40)
\node(q')(80,50){}
\drawedge[ELside=r](0',a'){$a$}
\drawedge(0',b'){$b$}
\drawedge[ELside=r](0',B'){$q$}
\drawedge[ELside=r](0',A'){$A$}
\drawedge[ELside=r,ELpos=50](A',q'){$B$}
\put(80,35){if $\delta(q, (a,b)) = p \in F$}
\end{picture}
\end{center} 
Assume that $\eval(\dA) = a_1 \cdots a_n$ and 
$\eval(\dB) = b_1 \cdots b_n$.
Consider the string 
$$
w = q_0q_0^{-1} a_1 a_1^{-1} \prod_{i=2}^n( A a_i a_i^{-1} \prod_{q \in Q} qq^{-1}) A BB^{-1} \prod_{i=0}^{n-1} (A^{-1} b_{n-i} b_{n-i}^{-1}) A^n. 
$$
It is easy to compute from $\dA$ and $\dB$ in polynomial time an SLP 
$\dC$ with $\eval(\dC) = w$.
The Munn tree $\MT(w)$ looks as follows:
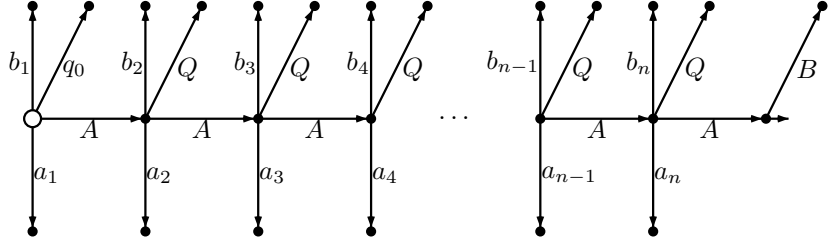
\begin{center}
\setlength{\unitlength}{1.5mm}
\begin{picture}(70,20)(0,0)
\gasset{Nadjust=n,ExtNL=y,NLdist=1,Nw=.7,Nh=.7,AHnb=1,Nfill=y,ELdist=0.1,AHLength=1,AHlength=1,linewidth=.2}
\node[Nw=1.5,Nh=1.5,Nfill=n](0)(0,10){}
\node(a1)(0,0){}
\node(b1)(0,20){}
\node(q0)(5,20){}
\node(1)(10,10){}
\node(a2)(10,0){}
\node(b2)(10,20){}
\node(q1)(15,20){}
\node(2)(20,10){}
\node(a3)(20,0){}
\node(b3)(20,20){}
\node(q2)(25,20){}
\node(3)(30,10){}
\node(a4)(30,0){}
\node(b4)(30,20){}
\node(q3)(35,20){}
\put(36,10){$\ldots$}
\node(n-1)(45,10){}
\node(an-1)(45,0){}
\node(bn-1)(45,20){}
\node(qn-1)(50,20){}
\node(n)(55,10){}
\node(an)(55,0){}
\node(bn)(55,20){}
\node(qn)(60,20){}
\node(n+1)(65,10){}
\drawline(65,10)(67,10)
\node(qn+1)(70,20){}
\drawedge[ELside=r](0,1){$A$}
\drawedge[ELside=r](1,2){$A$}
\drawedge[ELside=r](2,3){$A$}
\drawedge[ELside=r](n-1,n){$A$}
\drawedge[ELside=r](n,n+1){$A$}
\drawedge(0,a1){$a_1$}
\drawedge(0,b1){$b_1$}
\drawedge[ELside=r](0,q0){$q_0$}
\drawedge(1,a2){$a_2$}
\drawedge[ELside=l](1,b2){$b_2$}
\drawedge[ELside=r](1,q1){$Q$}
\drawedge(2,a3){$a_3$}
\drawedge[ELside=l](2,b3){$b_3$}
\drawedge[ELside=r](2,q2){$Q$}
\drawedge(3,a4){$a_4$}
\drawedge[ELside=l](3,b4){$b_4$}
\drawedge[ELside=r](3,q3){$Q$}
\drawedge(n-1,an-1){$a_{n-1}$}
\drawedge[ELside=l](n-1,bn-1){$b_{n-1}$}
\drawedge[ELside=r](n-1,qn-1){$Q$}
\drawedge(n,an){$a_n$}
\drawedge[ELside=l](n,bn){$b_n$}
\drawedge[ELside=r](n,qn){$Q$}
\drawedge[ELside=r](n+1,qn+1){$B$}
\end{picture}
\end{center}
Note that $\rho(w) \in \rho(K^*)$ if and only if 
$\rho(w) \in \rho(K_1)^{n-1} \rho(K_2)$. From this 
observation, it follows easily that $\rho(w) \in \rho(K^*)$
if and only if 
$\eval(\dA) \otimes \eval(\dB) \in L(\mcA)$.
\qed
\end{proof}
Let us now turn to an upper bound.

\begin{theorem} \label{RatMP in PSPACE}
The following problem belongs to $\PSPACE$:

\noindent
INPUT: An SLP $\dA$ over an alphabet $\Gamma \cup \Gamma^{-1}$ and 
an NFA $\mcA$ over the
alphabet $\Gamma \cup \Gamma^{-1}$.

\noindent
QUESTION: Does $\rho(\dA) \in \rho(L(\mcA))$ hold?
\end{theorem}
The proof of Thm.~\ref{RatMP in PSPACE} 
is based on tree automata techniques. Recall that 
a Munn tree $\MT(u)$ can be viewed as an edge labeled
tree. The node $\varepsilon$ can be made the root of the tree.
Such a rooted edge-labeled tree can be evaluated by a tree automaton.
Usually, tree automata work on node labeled trees, but this is only
a technicality, see the Appendix for precise definitions.
The proof of Thm.~\ref{RatMP in PSPACE} is based
on the following two Lemmas~\ref{constructing MT} and 
\ref{constructing TA}.

\begin{lemma}\label{constructing MT}
There is a $\PSPACE$-transducer, which computes
$\MT(\eval(\dA))$ for a given input SLP $\dA$.
\end{lemma}

\begin{proof}
For a given input word $u \in (\Gamma \cup \Gamma^{-1})^*$, the tree
$\MT(u)$ can be generated by a logspace transducer \cite{LoOn06IC}.
Moreover, the mapping $\dA \mapsto \eval(\dA)$ can be realized by a 
$\PSPACE$-transducer (Remark~\ref{remark:SLP}).
By Lemma~\ref{PSPACE}, we obtain a $\PSPACE$-transducer
realizing the mapping $\dA \mapsto \MT(\eval(\dA))$.
\qed
\end{proof}

\begin{lemma}\label{constructing TA}
There is a $\PSPACE$-transducer, which computes
from a given nondeterministic finite automaton $\mcA$ over the alphabet
$\Gamma \cup \Gamma^{-1}$ and a given SLP $\dA$  over the alphabet
$\Gamma \cup \Gamma^{-1}$ a 
nondeterministic tree automaton $\mcB = \mcB(\mcA,\dA)$  such that: 
$\rho(\eval(\dA)) \in \rho(L(\mcA))$ if and only if $\MT(\eval(\dA))$
is accepted by $\mcB$.
\end{lemma}
\medskip
\noindent
For the proof of Lemma~\ref{constructing TA} we need
some notations concerning multisets and tree automata.

\paragraph{\bf Multisets:}
A {\em multiset} over a set $A$ is a mapping $M : A \to \mathbb{N}$.
The {\em support} of $M$ is $\supp(M) = \{ a \in A \mid M(a) > 0\}$.
The {\em size} of $M$ is $|M| = \sum_{a \in A} M(a)$; we will only
consider multisets of finite size. 
For two multisets $M_1 : A \to \mathbb{N}$ and 
$M_2 : B \to \mathbb{N}$ we define the sum $M_1 + M_2$ as the 
following multiset over $A \cup B$: 
$$
(M_1+M_2)(x) = \begin{cases}
 M_1(x) & \text{ if } x \in A \setminus B \\
 M_2(x) & \text{ if } x \in B \setminus A \\
 M_1(x)+M_2(x)  & \text{ if } x \in A \cap B 
\end{cases}
$$
Addition of multisets is associative and commutative.
This allows us do define finite sums $\sum_{i \in I} M_i$
of multisets $M_i$, where $I$ is some finite set.
If $M$ is a multiset over $A$ and $f : A \to B$, then 
we define the multiset $f(M)$ over $B$ as follows:
$$
(f(M))(b) = \sum_{a \in f^{-1}(b)} M(a) .
$$
Clearly, $|M| = |f(M)|$.
We will consider multisets of words, i.e., multisets over
a set $\Sigma^*$. For a multiset 
$M$ over $\Sigma^*$ of finite size, we define the {\em total length}
$|\!| M |\!|$ of $M$ as
$|\!| M |\!| = \sum_{w \in \supp(M)} M(w) \cdot |w|$. 
Moreover, for a symbol $a \in \Sigma$ let
$|\!| M |\!|_a =  \sum_{w \in \supp(M)} M(w) \cdot |w|_a$
(where $|w|_a$ denotes the number of occurrences of 
the symbol $a$ in the word $w$).

\paragraph{\bf Tree automata:}
Let $\Theta$ be a finite alphabet.
A {\em $\Theta$-tree} is 
a finite and prefix-closed subset $T \subseteq \Theta^*$.
For $u \in T$ we define the tree 
$T \rest_u = \{ v \in \Theta^* \mid uv \in T\}$.
For $u \in T$ let us define $\out(u,T) = \{ a \in \Theta \mid ua \in T\}$.
A tree node $u \in T$ is a {\em leaf} of $T$ 
if $\out(u,T) = \emptyset$.
In the following, we will only consider
$(\Gamma \cup \Gamma^{-1})$-trees $T \subseteq \IRR(\Gamma)$
for a finite alphabet $\Gamma$.
Note that for a word $u \in (\Gamma \cup \Gamma^{-1})^*$,
the Munn tree $\MT(u) \subseteq \IRR(\Gamma)$ is such a $(\Gamma \cup \Gamma^{-1})$-tree.

A {\em tree automaton} (over the alphabet $\Theta$) is 
a triple $\mcB = (Q, \Delta, I)$, where
$Q$ is a finite set of states,
$\Delta \subseteq (\Theta \to_p Q) \times Q$
is the set of transitions,\footnote{$\Theta \to_p Q$
denotes the set of all partial mappings from $\Theta$ to $Q$.  
Our definition of a tree automaton is non-standard, but its suits
very well for our purpose.}  
and $I \subseteq Q$ is the set
of initial states. For a $\Theta$-tree
$T$, an {\em accepting run} of $\mcB$ on the tree 
$T$ is a mapping $\lambda : T \to Q$ such that:
\begin{itemize}
\item $\lambda(\varepsilon) \in I$
\item For every node $u \in T$, we have
$(f, \lambda(u)) \in \Delta$, where
$f$ is the partial mapping with
$\dom(f) = \out(u)$ and
$f(a) = \lambda(ua)$ for $a \in \out(u)$.
\end{itemize}
The tree language $L(\mcB)$ accepted by $\mcB$ is the 
set of all trees, for which there exists an accepting run.
For a state $q \in Q$ we let $L(\mcB,q) = L(Q,\Delta,\{q\})$, 
which is the language accepted by  the tree automaton 
that results from $\mcB$ by making $q$ the unique initial state.

\paragraph{\bf Loops in trees:}
We will consider loops in a $(\Gamma \cup \Gamma^{-1})$-tree $T\subseteq \IRR(\Gamma)$
that start and end in the root $\varepsilon$.  Such a loop can be identified with a word
over $\Gamma \cup \Gamma^{-1}$. Formally, an
{\em $\varepsilon$-loop} in $T$ is a word $\ell \in (\Gamma \cup \Gamma^{-1})^*$ such that
$r(\ell) = \varepsilon$ and $r(v) \in T$ for every prefix $v$ of $\ell$.
Let $\nodes(\ell) = \{ r(v) \mid v \preceq \ell \} \subseteq T$. 
This is the set of nodes that is obtained by starting in node $\varepsilon$ and 
walking along the unique $\ell$-labeled path.
Note that the empty word is an $\varepsilon$-loop.
We will be particularly interested in multisets over the sets of all $\varepsilon$-loops in $T$.
Let $\Lambda$ be such a multiset.
We say that $\Lambda$ covers $T$ if for every node $u \in T$
there exists an $\varepsilon$-loop $\ell \in \supp(\Lambda)$ such that
$u \in \nodes(\ell)$.

\medskip

\noindent
{\em  Proof of Lemma~\ref{constructing TA}.}
Let us fix the nondeterministic finite automaton $\mcA$ over the 
alphabet $\Gamma \cup \Gamma^{-1}$
and an SLP $\dA$ over the alphabet $\Gamma \cup \Gamma^{-1}$
for the rest of the proof. Let $u = \eval(\dA)$.
W.l.o.g. assume that $r(u) \neq \varepsilon$.
Moreover, we can assume that the last symbol $a_\ell$ of $u$
occurs only at the last position of $u$ (this can be enforced
by adding a new symbol to the alphabet $\Gamma$ which is appended to 
$u$). Note that $a_\ell$ is the last symbol of 
$r(u)$ as well.

By Thm.~\ref{munn}, $\rho(u) \in \rho(L(\mcA))$ if and only if 
there exists a path in the tree $\MT(u)$ from $\varepsilon$
to $r(u)$, which is labeled with a word from $L(\mcA)$, and 
which visits all nodes of $\MT(u)$.
Since $a_\ell$ occurs only at the last position of $u$,
the latter holds if and only if 
there exists a path in $\MT(u)$ from $\varepsilon$
to $\varepsilon$, which is labeled with a word from $L(\mcA) \cdot 
(\IRR(\Gamma) \cap  a_\ell^{-1}(\Gamma\cup\Gamma^{-1})^*)$, and 
which visits all nodes of $\MT(u)$.
A nondeterministic finite automaton $\mcA'$ for 
$L(\mcA) \cdot 
(\IRR(\Gamma) \cap  a_\ell^{-1}(\Gamma\cup\Gamma^{-1})^*)$
can be easily computed from $\mcA$. Hence, $\rho(u) \in \rho(L(\mcA))$
if and only if there exists an $\varepsilon$-loop $\ell$ in $\MT(u)$ 
with $\ell \in L(\mcA')$ and $\nodes(\ell) = \MT(u)$.
Let $\mcA' = (Q, \Gamma \cup \Gamma^{-1}, \delta, q_0, F)$.
A run of $\mcA'$ is a non-empty
word $r = (q_1, a_1, q_2) (q_2, a_2, q_3) \cdots (q_n, a_n, q_{n+1}) \in \delta^+$
of transition triples. Let $\lab(r) = a_1 a_2 \cdots a_n$. 
We also say that $r$ is a run from state $q_1$ to state $q_{n+1}$.
For states $p, q \in Q$ we denote with $L(\mcA',p,q)$ the set of all
words $\lab(r)$, where $r$ is a run from $p$ to $q$.
Moreover, for states $p,q \in Q$ and a $(\Gamma \cup \Gamma^{-1})$-tree
$T \subseteq \IRR(\Gamma)$, we denote
with $\lop(T,p,q)$ the set of all $\varepsilon$-loops $\ell$ in $T$ 
with $\ell \in L(\mcA',p,q)$.
Hence, we have to check whether for some $q \in F$ there exists $\ell \in \lop(\MT(u),q_0,q)$
with $\nodes(\ell) = T$.

We construct in $\PSPACE$ a tree automaton $\mcB$ that checks the latter property.
The idea of the construction can be explained as follows.
Let $n=|Q|$ be the number of states of $\mcA'$.
First of all, a pumping argument shows that 
if there exists a loop $\ell \in \lop(\MT(u),p,q)$
with $\nodes(\ell) = T$,  then
there exists a loop $\ell \in \lop(\MT(u),p,q)$ 
with $\nodes(\ell) = \MT(u)$ and 
$|\ell| \leq n \cdot |\MT(u)|^2$, see \cite[Proof of Theorem~3]{DiLoMi08}.
In the following, let $N = n \cdot |\MT(u)|^2$.
The set of states of $\mcB$ is the set of all
pairs $(M,s)$, where $0 \leq M \leq N$ and
$s$ is a multiset over $Q \times Q$ with $1 \leq |s| \leq N$.
Hence $\mcB$ contains at most
$(N+1)^{1+n^2}$ many states. The set of initial states
of $\mcB$ contains all states $(M,s_q)$ ($q \in F$) with
$M \leq N$ and   
$s_q(q_0,q) = 1$ and $s_q(p',q') = 0$ for all $(p',q') \in (Q \times Q)
\setminus \{(q_0,q)\}$.
The intuition behind a state $(M,s)$
is the following.  
Let $T \subseteq \IRR(\Gamma)$ be a $(\Gamma \cup \Gamma^{-1})$-tree.
We will have $T \in L(\mcB, (M,s))$ 
if and only if for all $p,q \in Q$
there exists a multiset $\Lambda_{p,q}$ over $\lop(T,p,q)$
such that the following holds, where $\Lambda = \sum_{p,q \in Q} \Lambda_{p,q}$:
\begin{itemize}
\item $|\Lambda_{p,q}| = s(p,q)$,
\item $|\!|\Lambda|\!| = M$,
\item $\Lambda$ covers $T$.
\end{itemize}
In other words: The tree
$T$ can be covered by 
$|s|$ many $\varepsilon$-loops  
of total length $M$,
where $s(p,q)$ many of these loops are labeled with 
a word from $L(\mcA',p,q)$.

It remains to construct the transition relation of $\mcB$.
For this, let us take 
a subset $\Omega \subseteq \Gamma \cup \Gamma^{-1}$,
let $(M,s)$ and $(M_a,s_a)$ ($a \in \Omega$)  be states of $\mcB$,
and let $f$ be the partial mapping with $f(a) = (M_a,s_a)$ for
$a \in \Omega$. We have to specify, whether the pair
$(f, (M,s))$ is a transition of $\mcB$. For this, the following definitions
are useful.

Let $\Delta$ be the set of all pairs $( (p,a,p'), (q,b,q') ) \in \delta \times \delta$
of transitions of $\mcA'$ such that $b = a^{-1}$.  
A non-empty word $w = (t_1,t'_1)(t_2,t'_2) \cdots (t_m,t'_m) \in
\Delta^+$ is {\em good}
if for all $1 \leq i < m$ the following holds: If 
$t'_i = (p,a,p')$ and $t_{i+1} = (q,b,q')$, then $p'=q$.
If $t_1 = (p,a,q)$ then we define $\first(w) = p$
and if $t'_m = (p,a,q)$ then we define $\last(w) = q$.
Moreover we denote with $\pi$ the projection
morphism from $\Delta^+$ to $(Q \times (\Gamma \cup \Gamma^{-1})
\times Q)^+$ with $\pi( (p,a,q), (p',a^{-1},q')) = (q,a,p')$.
The intuition behind good words is the following: Let $\ell$ be a non-empty
$\varepsilon$-loop in $T$ and fix a run $r$ of $\mcA'$ 
(the initial and final state of the run $r$ are not important)
with $\lab(r) = \ell$.
Then the $\varepsilon$-loop $\ell$ can factorized as
$\ell = \ell_1 \cdots \ell_m$, where each $\ell_i$ is a non-empty 
$\varepsilon$-loop which cannot be written as the concatenation of two non-empty
$\varepsilon$-loops. Thus, every $\ell_i$ is of the form $a \ell' a^{-1}$
for some $a \in \Gamma \cup \Gamma^{-1}$ and an $\varepsilon$-loop
$\ell'$ in $T\rest_a$. Hence, the run $r$ can be factorized as 
$r = r_1 \cdots r_m$, where $r_i$ is a run
of $\mcA'$ with $\lab(r_i) = \ell_i$ and $|r_i| \geq 2$.
Then, we obtain a good word 
\begin{equation}\label{def-g}
g(r) =  (t_1,t'_1) \cdots (t_m,t'_m),
\end{equation}
where the transition $t_i$ (resp. $t'_i$) is the first         
(resp. last) transition of the subrun $r_i$.
Moreover, for all $q,p \in Q$ and $a \in \Gamma \cup \Gamma^{-1}$, we 
define a multiset $\Lambda^a_{r,q,p}$ over $\lop(T\rest_a, q,p)$
as follows: $\Lambda^a_{r,q,p}(\ell')$ equals the number of indices $1 \leq i \leq m$
such that $\ell_i = a \ell' a^{-1}$, $t_i = (q', a, q)$ for some state $q'$, and 
$t'_i = (p,a^{-1},p')$ for some state $p'$.

Now, $(f, (M,s))$ is  a transition of $\mcB$ if and only if 
for all  $p,q \in Q$
there exist  multisets $W_{p,q} \subseteq \Delta^+$
of good words with the following properties, where $W = \sum_{p,q \in Q} W_{p,q}$:
\begin{itemize}
\item $|W_{p,q}| \leq s(p,q)$ and 
$|W_{p,q}| = s(p,q)$ in case $p \neq q$.
\item $\first(w) = p$ and $\last(w) = q$ for all $w \in \supp(W_{p,q})$.
\item $\pi(w) \in (Q \times \Omega \times Q)^+$ for every $w \in \supp(W)$.
\item $\sum_{a \in \Omega} M_a + 2 \cdot |\!|W|\!| = M$.
\item For all $a \in \Omega$, $q',p' \in Q$,
$s_a(q',p') = |\!|\pi(W)|\!|_{(q',a,p')}$. 
\end{itemize}
We claim that $\MT(u) \in L(\mcB)$ if and only if  
there exists an $\varepsilon$-loop $\ell \in L(\mcA')$ in $\MT(u)$ 
with $|\ell| \leq N$ and $\nodes(\ell) = \MT(u)$.
By the definition of the set of initial states of $\mcB$,
it suffices to prove the following more general claim:

\medskip
\noindent
{\em Claim:} Let $(M,s)$ be a state of $\mcB$
and let $T \subseteq \IRR(\Gamma)$ be a $(\Gamma \cup \Gamma^{-1})$-tree.
Then $T \in L(\mcB,(M,s))$  if and only if 
for all $p,q \in Q$
there exists a multiset $\Lambda_{p,q}$ over $\lop(T,p,q)$
such that the following holds, where $\Lambda = \sum_{p,q \in Q} \Lambda_{p,q}$:
\begin{itemize}
\item $|\Lambda_{p,q}| = s(p,q)$,
\item $|\!|\Lambda|\!| = M$,
\item $\Lambda$ covers $T$.
\end{itemize}
Both directions are shown by induction over the height of $T$.
First, assume that there exist  multisets $\Lambda_{p,q}$ over $\lop(T,p,q)$
such that the following holds, where $\Lambda = \sum_{p,q \in Q} \Lambda_{p,q}$:
\begin{itemize}
\item $|\Lambda_{p,q}| = s(p,q)$,
\item $|\!|\Lambda|\!| = M$,
\item $\Lambda$ covers $T$.
\end{itemize}
Let $\Lambda'_{p,q}$ be the multiset of all non-empty loops
in $\Lambda_{p,q}$. Formally, we set 
$\Lambda'_{p,q}(\varepsilon) = 0$ and
$\Lambda'_{p,q}(\ell) = \Lambda_{p,q}(\ell)$ if 
$\ell \neq \varepsilon$. Let $\Lambda' = \sum_{p,q \in Q} \Lambda'_{p,q}$.
We have $|\Lambda'_{p,q}| \leq s(p,q)$ and $|\Lambda'_{p,q}| = s(p,q)$
if $p \neq q$. 

We now define several multisets.
First of all, for all $p,q \in Q$ we can choose a multiset $R_{p,q}$
of runs of $\mcA'$ from $p$ to $q$ 
such that $\lab(R_{p,q}) = \Lambda'_{p,q}$.
Hence, $|R_{p,q}| = |\Lambda'_{p,q}|$.
Intuitively, we choose for every loop $\ell$ in $\Lambda'_{p,q}$ a run of $\mcA'$
from $p$ to $q$ with label $\ell$, where different runs may be chosen
for different occurrences of the same loop $\ell$ in the 
multiset $ \Lambda'_{p,q}$.  
Moreover, define the multiset $W_{p,q}$ over
$\Delta^+$ as $W_{p,q} = g(R_{p,q})$, where the mapping $g$
is defined in \eqref{def-g}.
Hence, $|R_{p,q}| = |W_{p,q}| = |\Lambda'_{p,q}|$.
Let $W = \sum_{p,q \in Q} W_{p,q}$ and $R = \sum_{p,q \in Q} R_{p,q}$.
Let $\Omega$ be the set of all symbols $a \in \Gamma\cup\Gamma^{-1}$
for which there exists $w \in \supp(W)$ such that $w$ contains a transition
pair of the form $( (p,a,q), (p',a^{-1},q') ) \in \Delta$.
Since  $\Lambda$ covers $T$, we must have $\Omega=\out(\varepsilon,T)$.
So far, we obtain the following properties for all $p,q \in Q$:
\begin{gather}
\text{$|W_{p,q}| \leq s(p,q)$ and 
$|W_{p,q}| = s(p,q)$ in case $p \neq q$} \label{cond-trans1}\\
\text{$\first(w) = p$ and $\last(w) = q$ for all $w \in \supp(W_{p,q})$} \label{cond-trans2} \\
\text{$\pi(w) \in (Q \times \Omega \times Q)^+$ for every $w \in \supp(W)$} \label{cond-trans3}
\end{gather}
Finally, for all $a \in \Omega$, $q,p \in Q$, 
we define the multiset $\Lambda^a_{q,p}$ over 
$\lop(T\rest_a, q,p)$ as follows:
$$
\Lambda^a_{q,p}(\ell) = \sum_{r \in \supp(R)} R(r) \cdot \Lambda^a_{r,q,p}(\ell).
$$
Let $\Lambda^a = \sum_{q,p \in Q} \Lambda^a_{q,p}$.
Since $\Lambda$ covers $T$, we get:
\begin{equation} \label{ind1}
\text{$\Lambda^a$ covers  $T\rest_a$ for all $a \in \Omega$.}
\end{equation}
We now define for every $a \in \Omega$ a state $(M_a, s_a)$ of the 
tree automaton $\mcB$ as follows:
For all $q,p \in Q$, let
\begin{equation} \label{cond-trans5}
s_a(q,p) = |\!|\pi(W)|\!|_{(q,a,p)}.
\end{equation}
This easily implies
\begin{equation}  \label{ind2} 
s_a(q,p) = |\Lambda^a_{q,p}| .
\end{equation}
Moreover, for all $a \in \Omega$, we define
\begin{equation}  \label{ind3}
M_a = |\!|\Lambda^a|\!| .
\end{equation}
This implies
\begin{equation} \label{cond-trans4}
M = \sum_{a\in \Omega} M_a + 2 \cdot |\!|W|\!|.
\end{equation}
Here, the factor two comes from the fact that each symbol in 
a word from $W$ is a pair of transitions.
Hence, using  \eqref{ind1}, \eqref{ind2}, \eqref{ind3},
and induction over the tree height, we get $T\rest_a \in L(\mcB,(M_a,s_a))$.
Moreover, by \eqref{cond-trans1}, \eqref{cond-trans2}, \eqref{cond-trans3}, \eqref{cond-trans5}, and
\eqref{cond-trans4}, the pair $(f,(M,s))$, where $f$ is the partial mapping with
$\dom(f) = \Omega$ and $f(a) = (M_a,s_a)$ for $a \in \Omega$,
is a transition of $\mcB$. Hence, we have indeed $T \in L(\mcB,(M,s))$.

The other direction of the above claim can be shown similarly.
It remains to argue that $\mcB$ can be computed by a $\PSPACE$-transducer
with input $\mcA, \dA$.
The automaton $\mcA'$ can be computed in polynomial time from $\mcA$.
Since $\MT(\eval(\dA))$ can be computed in $\PSPACE$ from 
$\dA$ by Lemma~\ref{constructing MT}, we can compute the binary representation of the number
$N$ in $\PSPACE$ as well. Note that the number $N$ is exponentially
bounded in the size of the input $\mcA, \dA$.
Hence, every transition of $\mcB$ can be described with polynomially
many bits. It suffices to show that one can check in $\PSPACE$,
whether a given pair $(f,(M,s))$ is indeed a transition of $\mcB$.
This follows easily from the definition of the transitions of $\mcB$.
One can guess the multisets $W_{p,q}$ ($p,q \in Q$), but instead of storing
these sets one only stores the current transition pair and thereby
accumulates the values $|W_{p,q}|$, $|\!|W|\!|$, and $|\!|W|\!|_{(q',a,p')}$ 
for all $a \in \Omega$ and $q',p' \in Q$.
\qed

\medskip

\noindent
{\em Proof of Thm.~\ref{RatMP in PSPACE}.}
We apply Lemma~\ref{PSPACE}, where
$f : (\dA, \mcA) \mapsto (\MT(\eval(\dA)), \mcB(\mcA,\dA))$ and
$L$ is the uniform membership problem for tree
automata, i.e., the set of all pairs $(T,\mcB)$, where
$T$ is a tree and $\mcB$ is a tree automaton that accepts $T$.
By \cite{Loh01rta}, $L$ belongs to  $\LOGCFL$
and hence to $\POLYLOGSPACE$ \cite{LeStHa65}. 
Moreover, the mapping $f$ can be computed by a $\PSPACE$-transducer by Lemma~\ref{constructing MT}
and \ref{constructing TA}.
\qed

\def\cprime{$'$}

\end{document}